\documentclass[11pt,a4paper]{amsart}
\usepackage[T1]{fontenc}
\usepackage{mathrsfs}
\usepackage{syntonly}
\usepackage{amsmath}
\usepackage{amsthm}
\usepackage{amsfonts}
\usepackage{amssymb}
\usepackage{latexsym}
\usepackage{amscd,amssymb,amsopn,amsmath,amsthm,graphics,amsfonts,mathrsfs,accents,enumerate,verbatim,calc}
\usepackage{mathtools}
\usepackage[colorlinks=true,linkcolor=red,citecolor=blue]{hyperref}
\usepackage{booktabs,graphicx}
\usepackage[all]{xy}
\usepackage{enumitem}
\allowdisplaybreaks[4] \footskip=15pt
\renewcommand{\uppercasenonmath}[1]{}

\numberwithin{equation}{section} \theoremstyle{plain}
\newtheorem{theorem}{Theorem}[section]
\newtheorem{corollary}[theorem]{Corollary}
\newtheorem{lemma}[theorem]{Lemma}

\theoremstyle{definition}
\newtheorem{definition}[theorem]{Definition}
\newtheorem{example}[theorem]{Example}
\newtheorem{construction}[theorem]{Construction}

\newtheorem{remark}[theorem]{Remark}

\newtheorem*{ack*}{ACKNOWLEDGEMENTS}

\newcommand{\Cpx}[4]{\mathbf{C}(#1)^{[#2,#3]}_{\mathcal{#4}}}

\newcommand{\Pres}[2]{\Cpx{A}{-#1}{0}{#2}}
\newcommand{\Copres}[2]{\Cpx{A}{0}{#1}{#2}}

\newcommand{\PresP}[1]{\Pres{#1}{P,E}}
\newcommand{\PresFP}[1]{\Pres{#1}{FP,E}}
\newcommand{\CopresE}[1]{\Copres{#1}{I,E}}

\newcommand{\oo}{\otimes}
\newcommand{\pf}{\noindent\begin {proof}}

\newcommand{\s}{\stackrel}

\newcommand{\Mod}{\mbox{\rm Mod}}

\newcommand{\Ext}{{\rm Ext}}
\newcommand{\Hom}{{\rm Hom}}

\newcommand{\pp}{\mathfrak{p}}

\newcommand{\coker}{{\rm coker}}

\newcommand{\Ker}{{\mathrm{Ker}}}

\title{On silting complexes associated to $n$-silting modules}

\author{Michal Hrbek}
\address[M. Hrbek]{Institute of Mathematics, Czech Academy of Sciences, \v{Z}itn\'{a} 25, 115 67 Prague, Czech Republic}
\email{hrbek@math.cas.cz}

\author{Jiangsheng Hu}
 \address[J. Hu]{School of Mathematics, Hangzhou Normal University, Hangzhou 311121, China}
 \email{hujs@hznu.edu.cn}

\author{Rongmin Zhu}
\address[R. Zhu]{School of Mathematical Sciences, Huaqiao University, Quanzhou 362021, P. R. China}
\email{rongminzhu@hotmail.com}


\begin{document}

\begin{abstract}
We show that any $(n+1)$-term silting complex whose intermediate cohomology vanishes gives rise to an $n$-silting module, as recently introduced by Mao. Specializing to commutative noetherian rings, we show that this assignment induces a bijection on the respective equivalence classes. Furthermore, we prove in the same setting that the $n$-silting modules always correspond to a tilting complex, that is, the associated t-structure is of derived type. We use this to exhibit new examples of tilting complexes in the setting of Commutative Algebra and also to show that the finite type property for $n$-silting modules, as formulated by Mao, can in general fail.
\end{abstract}

\maketitle

\section*{Introduction}

Tilting modules originated as a fundamental notion in Representation Theory.
The classical tilting theory started in the context of finitely generated modules over Artin algebras due to Happel and Ringel \cite{ha82}. Beginning with Miyashita \cite{mi86}
and in turn Colby and Fuller \cite{fu90}, finitely generated tilting modules over arbitrary
rings were studied. After that, generalizations to infinitely generated tilting and
cotilting modules over arbitrary associative rings were investigated first by Colpi
and Trlifaj \cite{tr95}, Colpi, Tonolo and Trlifaj \cite{tr97}, and Angeleri H\"{u}gel and Coelho \cite{li01}.

Later, the notion of silting modules was introduced by Angeleri H\"{u}gel, Marks, and Vit\'{o}ria \cite{li16} as a common generalization of tilting modules over an arbitrary ring and the support $\tau$-tilting modules over a finite dimensional algebra, as introduced by Adachi, Iyama and Reiten \cite{iya14}.
Silting modules are in bijection with 2-term silting complexes and with certain t-structures and co-t-structures in the derived module category.
The dual notion, i.e., cosilting modules, as a generalization of cotilting modules,
were studied by Breaz and Pop \cite{br17} and Zhang and Wei \cite{zw17}. Recently, Mao \cite{mao22} introduced the concepts of
$n$-silting and $n$-cosilting modules, which
not only generalize $n$-tilting and $n$-cotilting modules,
but also generalize silting and cosilting modules.

In \cite[4.8]{mao22}, it was asked if, for any $n$-silting module with respect to a projective $n$-presentation $\Sigma$, the complex $\Sigma$ constitutes an $(n+1)$-term silting complex, generalizing the known bijection from the $n=1$ case. We give a partial answer to this in Theorem~\ref{thm1}, by first showing that any $(n+1)$-term silting complex $\Sigma$ with vanishing intermediate cohomology (i.e., $H^i(\Sigma)=0$ for all $-n<i<0$) is a projective $n$-presentation of an $n$-silting module. Furthermore, we show in Theorems~\ref{thm1} and \ref{thm1co} that the converse is also true if $A$ satisfies certain condition, which is always satisfied in case $A$ is left hereditary or commutative noetherian. We also establish the dual claim for $n$-cosilting modules. We also give a characterization in Lemma~\ref{equiv-conditions} of the situation in which an $n$-silting module gives rise to a silting complex in terms of certain condition on the kernel of the induced cotorsion pair in a suitable category of representations over the base ring, based on the work of Marks and Vitória \cite{ma18}.

Finally, we specialize to the setting of Commutative Algebra, where the base ring $A$ is a commutative noetherian ring, and provide more substantial motivation for studying $n$-(co)silting modules by both showing that the corresponding (co)silting complexes enjoy certain nice properties and by providing explicit constructions of $n$-(co)silting modules which are neither $n$-(co)tilting modules nor $1$-(co)silting modules. More precisely, we show in Theorem~\ref{cotilting-complex} that the (co)silting complex corresponding to any $n$-(co)silting module is, in this setting, always a (co)tilting complex, that is, the associated t-structure is of derived type. This generalizes the previously known case for $n=1$ of Pavon and Vitória \cite[Corollary 5.12]{pv21}, see also \cite[Theorem 6.18]{hns24} to arbitrary $n$. Note that, in general, determining whether a silting complex is a tilting complex is a subtle problem, as demonstrated in this setting e.g. in \cite[\S 7]{hns24}. We then proceed to construct the promised examples over local rings which fail to be Cohen--Macaulay, but enjoy vanishing of intermediate local cohomology, these were classically constructed by Sharp \cite{sha75}. In particular, the cosilting complex constructed in Example~\ref{E:e2} is neither a cotilting module nor is the induced t-structure restrictable in the sense of \cite{pv21}, and thus it demonstrates Theorem~\ref{cotilting-complex} as a new way to establish the derived type.

The finite type property was established for $n$-tilting modules in celebrated results of Bazzoni-Herbera \cite{bh07} and Bazzoni-Šťovíček \cite{sto07}. The generalization for silting complexes was obtained by Marks and Vitória \cite{ma18}. In \cite{mao22}, Mao formulated an analogous finite type property for $n$-silting modules in terms of $n$-projective presentations consisting of finitely generated projective modules, and asked if this property is enjoyed by all $n$-silting modules. We show in \S\ref{counterexample} that our construction in the setting of commutative noetherian rings provides a counterexample, that is, an $n$-silting module which is not of finite type in this sense.

\renewcommand{\thetheorem}{\arabic{section}.\arabic{theorem}}
\bigskip
\section{Preliminaries}
\subsection{Basic notation}
Throughout, let $A$ be an associative ring with unit.
Denote by $\mathrm{Mod}(A)$ the category of left $A$-modules.
If not stated otherwise, by an $A$-module, we always mean a left $A$-module
and right $A$-modules are considered as left $A^{op}$-modules, where $A^{op}$ is the opposite ring of $A$.
Denote by $\mathrm{Proj}(A)$ (resp. $\mathrm{proj}(A)$) the category of  (resp. finitely generated) projective $A$-modules, by $\mathrm{Inj}(A)$ the category of injective $A$-modules.
Denote by $\mathrm{pd}_{A}(M)$ (resp. $\mathrm{id}_A(M)$) the projective (resp. injective) dimension of $M$.
For a right $A$-module $M$, we denote by $M^{+}$ the character module
$\Hom_{\mathbb{Z}}(M, \mathbb{Q}/\mathbb{Z}) \in \mathrm{Mod}(A)$.
Note that $A^{+}$, the character module of the right regular module $A_{A}$, is an injective cogenerator of $\mathrm{Mod}(A)$.

Denote by $\mathbf{C}(A)$ and $\mathbf{D}(A)$ the category of all cochain complexes and the (unbounded) derived category of cochain complexes in $\mathrm{Mod}(A)$, respectively. Throughout the paper, we will consider several useful subcategories of $\mathbf{C}(A)$, which we will sometimes consider naturally also as subcategories of $\mathbf{D}(A)$: 

\begin{itemize}
\item For a couple of integers $a \leq b$, we let $\Cpx{A}{a}{b}{}$ denote the subcategory of $\mathbf{C}(A)$ consisting of complexes concentrated in degrees $a,a+1,\ldots,b$, that is, complexes of the form 
$$X = (\cdots \to 0 \to  0 \to X^{a} \to X^{a+1} \to \cdots \to X^b \to 0 \to 0 \to \cdots).$$
\item $\Cpx{A}{a}{b}{P}$ consists of those cochain complexes $X \in \Cpx{A}{a}{b}{}$ 
whose components $X_i$ are projective $A$-modules.
\item $\Cpx{A}{a}{b}{FP}$ consists of those cochain complexes $X \in \Cpx{A}{a}{b}{}$ whose components $X_i$ are finitely generated projective $A$-modules.
\item $\Cpx{A}{a}{b}{I}$ consists of those cochain complexes $X \in \Cpx{A}{a}{b}{}$ whose components $X_i$ are injective $A$-modules.
\item $\Cpx{A}{a}{b}{E}$ consists of those cochain complexes $X \in \Cpx{A}{a}{b}{}$ whose \textit{intermediate cohomology vanishes}, that is, such that $H^i(X) = 0$ for all $a<i<b$.
\item Finally, we shall use the natural shorthand notation $\Cpx{A}{a}{b}{P,E} = \Cpx{A}{a}{b}{P} \cap \Cpx{A}{a}{b}{E}$ and similarly we define $\Cpx{A}{a}{b}{I,E}$ and $\Cpx{A}{a}{b}{FP,E}$.
\end{itemize}   

Note that the objects of $\PresP{n}$ are nothing else than the projective $n$-presentations of $A$-modules, that is, $\Sigma \in \PresP{n}$ is the projective $n$-presentation of the $A$-module $T=H^0(\Sigma)$. Similarly, the objects of $\CopresE{n}$ are injective $n$-copresentations of $A$-modules, that is, $\Omega \in \CopresE{n}$ is the injective $n$-copresentation of the $A$-module $C=H^0(\Omega)$.

For an object $X \in \mathbf{D}(A)$, we denote by $\mathrm{Add}(X)$ the class of all objects which are isomorphic to direct summands
of direct sums of copies of $X$ and for an $A$-module $M$ denote by $\mathrm{Gen}(M)$ the class of all $M$-generated modules, i.e. all epimorphic
images of modules in $\mathrm{Add}(M)$. Dually, we can define $\mathrm{Prod}(M)$ and $\mathrm{Cogen}(M)$. For an $A$-module $M$, denote by $\mathrm{Pres}^{n}(M)$ the class
of $A$-modules $N$ for which there is an exact sequence of the form
$$C_{n}\rightarrow C_{n-1}\rightarrow \cdots \rightarrow C_{0}\rightarrow N \rightarrow 0$$
with each $C_{i} \in \mathrm{Add}(M)$. In fact, $\mathrm{Gen}(M) = \mathrm{Pres}^{0}(M)$. For an $A$-module $N$, denote by $\mathrm{Copres}^{n}(N)$ the class of  $A$-modules $M$ for which there is an exact sequence of the form
$0\rightarrow M \rightarrow C^{0} \rightarrow\cdots\rightarrow C^{n-1} \rightarrow C^{n}$ with each $C^{i} \in \mathrm{Prod}(N)$.

Given a class $\mathcal{C}$ of $A$-modules and an $A$-module $M$, set
$$\mathcal{C}^{\perp_{1}}= \{N\in \mathrm{Mod}(A)\mid \Ext^{1}_{A}(C, N) = 0
~\forall  C\in \mathcal{C}\}, $$
$$^{\perp_{1}}\mathcal{C}= \{N\in \mathrm{Mod}(A)\mid \Ext^{1}_{A}(N,C) = 0
~\forall  C\in \mathcal{C}\},$$
$$M^{\perp_{\infty}} = \{N \in \mathrm{Mod}(A) \mid \Ext_A^i(M,N) = 0 ~\forall i>0\},$$
$${}^{\perp_{\infty}}M = \{N \in \mathrm{Mod}(A) \mid \Ext_A^i(N,M) = 0 ~\forall i>0\}.$$
Recall that a pair $(\mathcal{A}, \mathcal{B})$ of full subcategories in $\mathrm{Mod}(A)$ is called a
\emph{cotorsion pair} if $\mathcal{A}^{\perp_{1}}= B$ and $\mathcal{B} = {^{\perp_{1}}\mathcal{A}}$.

Given a class $\mathcal{C}$ of $A$-modules and an $A$-module $M$. Recall that a homomorphism $f\in \Hom_{A}(C, M)$ with $C\in\mathcal{C}$  is said to be a \emph{special $\mathcal{C}$-precover} of  $M$ if $f$ is surjective and $\ker(f)\in \mathcal{C}^{\perp_{1}}$. Dually, a homomorphism $g\in \Hom_{A}(M, C)$ with $C\in\mathcal{C}$ is said to be a \emph{resp. special $\mathcal{C}$-preenvelope} of $M$ if $g$ is injective and $\coker(g)\in{^{\perp_{1}}\mathcal{C}}$.
The  class $\mathcal{C}$ is called \textit{special precovering} (resp. \textit{special preenveloping}) if every $A$-module has a special  $\mathcal{C}$-precover  (special $\mathcal{C}$-preenvelope). A cotorsion pair $(\mathcal{A}, \mathcal{B})$ is said to be \emph{complete} if  $\mathcal{A} $ is special precovering and $\mathcal{B} $ is special preenveloping.

Given an object $X \in \mathbf{D}(A)$, we define the following subcategories of the derived category:
$$X^{\perp_{>0}} = \{Y \in \mathbf{D}(A) \mid \Hom_{\mathbf{D}(A)}(X,Y[i]) = 0~\forall i>0\},$$
$$X^{\perp_{\leq 0}} = \{Y \in \mathbf{D}(A) \mid \Hom_{\mathbf{D}(A)}(X,Y[i]) = 0~\forall i\leq 0\},$$
$${}^{\perp_{>0}}X = \{Y \in \mathbf{D}(A) \mid \Hom_{\mathbf{D}(A)}(Y,X[i]) = 0~\forall i>0\},$$
$${}^{\perp_{\leq 0}}X = \{Y \in \mathbf{D}(A) \mid \Hom_{\mathbf{D}(A)}(Y,X[i]) = 0~\forall i\leq 0\},$$
$$X^{\perp_{\mathbb{Z}}} = X^{\perp_{>0}} \cap X^{\perp_{\leq 0}} \text{, and } {}^{\perp_{\mathbb{Z}}}X = {}^{\perp_{>0}}X \cap {}^{\perp_{\leq 0}}X.$$

Recall that a pair $\mathcal{(U, V)}$ of subcategories of $\mathbf{D}(A)$ is said to be a
\emph{torsion pair} provided that the following conditions hold:
\begin{itemize}
\item [$(i)$]
    Both $\mathcal{U}$ and $\mathcal{V}$ are closed under direct summands.
\item [$(ii)$]
    $\mathcal{U}^{\perp_{0}} = \mathcal{V}$ and $\mathcal{U} = {^{\perp_{0}}} \mathcal{V}$.
\item [$(iii)$]
    Every object $X \in \mathbf{D}(R)$ lies in a triangle
    $$U \rightarrow X \rightarrow V \rightarrow U[1]$$
    in $\mathbf{D}(R)$ with $U \in \mathcal{U}$ and $V \in \mathcal{V}$.
\end{itemize}
Recall that a torsion pair $\mathcal{(U, V)}$ is called a \emph{t-structure}
(resp. \emph{co-t-structure}) if in addition
$\mathcal{U}$ is closed under positive (resp. negative) shifts, that is, $\mathcal{U}[1] \subseteq \mathcal{U}$ (resp. $\mathcal{U}[-1] \subseteq \mathcal{U}$). Recall that the derived category $\mathbf{D}(A)$ always admits the \textit{standard t-structure} $(\mathbf{D}^{\leq 0},\mathbf{D}^{>0})$ given by the standard cochain cohomology, that is:
$$\mathbf{D}^{\leq 0} = \{X \in \mathbf{D}(A) \mid H^i(X)=0 ~\forall i>0\},$$
$$\mathbf{D}^{> 0} = \{X \in \mathbf{D}(A) \mid H^i(X)=0 ~\forall i\leq 0\}.$$

\subsection{$n$-tilting modules and $n$-cotilting modules}

Recall that an $A$-module $T$ is \emph{$n$-tilting} \cite{ba04, tr12} if it satisfies the following three conditions:
\begin{itemize}
  \item [$(T1)$] $\mathrm{pd}_{A}(T) \leq n$;
  \item [$(T2)$] $T^{(\kappa)} \in T^{\perp_\infty}$ for any cardinal $\kappa$;
  \item [$(T3)$] there exist an integer $r\geq 0$ and an exact sequence
  $$0\rightarrow A \rightarrow  T^{0}\rightarrow T^{1}\rightarrow\cdots\rightarrow T^{r}\rightarrow 0$$
   with each $T^{i} \in \mathrm{Add}(T)$ for any $0 \leq i\leq r$.
\end{itemize}
If an $n$-tilting module $T$ satisfies the following stronger version of
$(T1)$: $T$ has a projective resolution $0\rightarrow P^{n}\rightarrow \cdots \rightarrow P^{0}\rightarrow T\rightarrow 0$ with each $P^{i}$ finitely generated, then $T$ is called a \emph{classical $n$-tilting module}.
By \cite[Theorem 3.11]{ba04}, $T$ is $n$-tilting if and only if $\mathrm{Pres}^{n-1}(T)=T^{\perp_{\infty}}$.

Recall that $C$ is an \emph{$n$-cotilting module}  \cite{ba04, tr12} if it satisfies the following three conditions:
\begin{itemize}
  \item [$(C1)$] $\mathrm{id}_{A}(C) \leq n$;
  \item [$(C2)$] $C^{\kappa} \in {}^{\perp_\infty}C$ for any cardinal $\kappa$;
  \item [$(C3)$] there exist an integer $r\geq 0$ and an exact sequence
  $$0\rightarrow C^{r}\rightarrow  C^{r-1}\rightarrow\cdots\rightarrow C^{0}\rightarrow A^{+}\rightarrow 0$$
   with each $C^{i} \in \mathrm{Prod}(C)$ for any $0 \leq i\leq r$.
\end{itemize}
By \cite[Theorem 3.11]{ba04}, $C$ is $n$-cotilting if and only if $\mathrm{Copres}^{n-1}(C)={^{\perp_{\infty}}C}$. 

Any $n$-tilting module $T$ gives rise to a \textit{tilting cotorsion pair} $\mathfrak{C} = ({}^{\perp_\infty}(T^{\perp_\infty}),T^{\perp_\infty})$ whose \textit{kernel} $\Ker(\mathfrak{C}) = {}^{\perp_\infty}(T^{\perp_\infty}) \cap T^{\perp_\infty}$ is equal to $\mathrm{Add}(T)$. Dually, an $n$-cotilting module $C$ gives rise to a \textit{cotilting cotorsion pair} $\mathfrak{C} = ({}^{\perp_\infty}C,({}^{\perp_\infty}C)^{\perp_\infty})$ such that $\Ker(\mathfrak{C}) = \mathrm{Prod}(C)$.

\subsection{\bf Silting and cosilting complexes}
Let $\mathbf{K}^{b}(\mathrm{Proj}(A))$ and $\mathbf{K}^{b}(\mathrm{Inj}(A))$ denote the homotopy category of bounded complex of complexes with projective (resp. injective) components, both viewed naturally as subcategories of $\mathbf{D}(A)$. Recall that a complex $X$ in $\mathbf{D}(A)$ is said to be
\begin{itemize}
  \item \emph{silting} if $\Hom_{\mathbf{D}(A)}(X, X^{(J)}[i]) = 0$ for all $i > 0$ and all sets $J$,
  and $\mathrm{thick}(\mathrm{Add}(X)) =\mathbf{K}^{b}(\mathrm{Proj}(A))$; it is moreover said to be
  \emph{$(n+1)$-term silting} if $X$ lies in $\PresP{n}(A)$.
  \item  \emph{cosilting} if $\Hom_{\mathbf{D}(A)}(X^{J} , X[i]) = 0$ for all $i > 0$ and all sets $J$,
  and $\mathrm{thick}(\mathrm{Prod}(X)) =\mathbf{K}^{b}(\mathrm{Inj}(A))$; it is moreover said to be
  \emph{$(n+1)$-term cosilting} if $X$ lies in $\CopresE{n}(A)$.
\end{itemize}

Recall from \cite{pv18} that any silting complex $X$ comes attached with a t-structure in $\mathbf{D}(A)$ of the form $(X^{\perp_{>0}},X^{\perp_{\leq 0}})$, called the associated \textit{silting t-structure}. We say that two silting complexes are equivalent if they give rise to the same t-structure. It turns out that two silting complexes $X,X'$ are equivalent precisely if $\mathrm{Add}(X)=\mathrm{Add}(X')$, see \cite[Lemma 4.5]{pv18}. Similarly, any cosilting complex gives rise to a \textit{cosilting t-structure} $({}^{\perp_{>0}}X,{}^{\perp_{\leq 0}}X)$, leading to the dual notion of equivalence of cosilting complexes.

For the purposes of the last section, we also recall the notions of tilting and cotilting complexes. A silting complex $X$ (resp. a cosilting complex $X$) is called a \textit{tilting complex} (resp. a \textit{cotilting complex}) provided that $X \in X^{\perp_{< 0}} = (X^{\perp_{\leq 0}}) [1]$ (resp. $X \in {}^{\perp_{< 0}}X = ({}^{\perp_{\leq 0}}X)[-1]$). These notions are important because they characterize precisely when the silting t-structure (resp. the cosilting t-structure) is of derived type. In fact, in this situation, the derived category $\mathbf{D}(A)$ is triangle equivalent to the heart of the t-structure, see \cite{pv18} and \cite{vir18}.

It is a standard fact that a module $M \in \mathrm{Mod}(A)$ is $n$-(co)tilting module if and only if it is an $(n+1)$-term (co)silting complex as an object of the derived category $\mathbf{D}(A)$, and then it is automatically a (co)tilting complex. In general however, determining whether a (co)silting complex is a (co)tilting complex is a subtle problem, see \cite{hns24}.

Finally, the following characterization of silting and cosilting complexes will be useful for us.

\begin{lemma}\label{silting-complex-generate}\cite[Theorem 3.11]{bhm25}
  A complex $X \in \mathbf{K}^{b}(\mathrm{Proj}(A))$ is silting if and only if the following two conditions hold:
  \begin{enumerate}
    \item[$(i)$] $\mathrm{Add}(X) \subseteq X^{\perp_{>0}}$,
    \item[$(ii)$] $X^{\perp_{\mathbb{Z}}} = 0$.
  \end{enumerate}
  A complex $X \in \mathbf{K}^{b}(\mathrm{Inj}(A))$ is cosilting if and only if the following two conditions hold:
  \begin{enumerate}
    \item[$(i)$] $\mathrm{Prod}(X) \subseteq {}^{\perp_{>0}}X$,
    \item[$(ii)$] $  {}^{\perp_{\mathbb{Z}}}X = 0$.
  \end{enumerate}
\end{lemma}

\subsection{Silting complexes as tilting modules}

In the following, we recall the interpretation of (co)silting complexes as (co)tilting module in a suitable module category as presented in \cite{ma18}. In order to
do this, recall first that the subcategory $\Cpx{A}{a}{b}{}$ is in fact a module category over a suitable ring. Indeed, set $n = b - a$ and let us consider the quiver
 $$Q_{n+1} :\overset{0}{\bullet}\rightarrow \overset{1}{\bullet} \rightarrow \cdots\rightarrow \overset{n-1}{\bullet}\rightarrow \overset{n}{\bullet},$$
 as well as the bound quiver $A$-algebra $AQ_{n+1}/I$, where $I$ is the ideal generated by all paths of length 2. It is standard to check that the functor $\Theta: \mathrm{Mod}(AQ_{n+1}/I) \rightarrow \mathbf{C}(A)$ which sends a bound representation of $Q_{n+1}$ to a chain complex concentrated in degrees $a,a+1,\ldots,b$ induces an equivalence $\mathrm{Mod}(AQ_{n+1}/I) \cong \Cpx{A}{a}{b}{}$. For any symbol $\star \in \{\mathcal{P},\mathcal{I},\mathcal{FP},\mathcal{(P,E)},\mathcal{(I,E)},\mathcal{(FP,E)}\}$, consider the subcategory 
 $$\mathrm{Mod}_{\star}(AQ_{n+1}/I) = \{M \in \mathrm{Mod}(AQ_{n+1}/I) \mid \Theta(M) \in \Cpx{A}{a}{b}{\star}\}.$$
 
 We also implicitly consider the functor $\Psi: \mathrm{Mod}(AQ_{n+1}/I) \rightarrow \mathbf{D}(A)$ which is defined as the composition of $\Theta$ with the canonical functor $\mathbf{C}(A)\rightarrow \mathbf{D}(A)$. 

 We mark down a crucial property of this identification, see \cite[Lemma 2.4]{ma18}. Let $M \in \mathrm{Mod}(AQ_{n+1}/I)$, $P \in \mathrm{Mod}_\mathcal{P}(AQ_{n+1}/I)$, and $I \in \mathrm{Mod}_\mathcal{I}(AQ_{n+1}/I)$, then we have for any $i>0$ the isomorphisms:
  $$\Ext_{AQ_{n+1}/I}^j(P,M) \cong \Hom_{\mathbf{D}(A)}(\Psi(P),\Psi(M)[j]),$$
  $$\Ext_{AQ_{n+1}/I}^j(M,I) \cong \Hom_{\mathbf{D}(A)}(\Psi(M),\Psi(I)[j]).$$

By \cite[Theorem 2.10]{ma18}, we obtain the following bijections.
\begin{theorem}\label{mv-bijection}
  Setting $a=-n$ and $b=0$, the functor $\Psi$ induces a bijection:
   $$\left \{ \begin{tabular}{ccc} \text{ $n$-tilting modules in $\mathrm{Mod}(AQ_{n+1}/I)$} \\
   \text{which belong to $\mathrm{Mod}_{\mathcal{P}}(AQ_{n+1}/I)$} \\ \text{up to equivalence} \end{tabular}\right \}  \xleftrightarrow{1-1}  \left \{ \begin{tabular}{ccc} \text{} \\ \text{$(n+1)$-silting complexes in $\mathbf{D}(A)$}   \\ \text{up to equivalence} \end{tabular}\right \}.$$
   Setting $a=0$ and $b=n$, the functor $\Psi$ induces a bijection:
   $$\left \{ \begin{tabular}{ccc} \text{ $n$-cotilting modules in $\mathrm{Mod}(AQ_{n+1}/I)$} \\ \text{which belong to $\mathrm{Mod}_{\mathcal{I}}(AQ_{n+1}/I)$} \\\text{up to equivalence} \end{tabular}\right \}  \xleftrightarrow{1-1}  \left \{ \begin{tabular}{ccc} \text{} \\ \text{$(n+1)$-cosilting complexes in $\mathbf{D}(A)$}  \\ \text{up to equivalence} \end{tabular}\right \}.$$
\end{theorem}

\subsection{$n$-silting modules and $n$-cosilting modules}\label{subsec:nsilt}
Now we recall the recent notion of $n$-silting and $n$-cosilting modules as recently introduced by Mao \cite{mao22}.

Let $\Sigma$ be an object $P_{n}\rightarrow P_{n-1}\rightarrow \cdots  \rightarrow P_{0}$ in $\PresP{n}$. Let
\begin{center}
$\mathcal{D}_{\Sigma}=\left\{X \in \mathrm{Mod}(A)~~ \bigg| ~~{\left.
\begin{aligned}
&\text{\rm The induced complex } \Hom_{A}(P_{0}, X)\rightarrow
\cdots\rightarrow \Hom_{A}(P_{n-1}, X) \\
&\rightarrow\Hom_{A}(P_{n}, X)\rightarrow 0\text{ is exact} \end{aligned}
\right. }
                                                               \right\}.$
\end{center}

In another words, observe that we have 
$$\mathcal{D}_{\Sigma} = \{X \in \mathrm{Mod}(A) \mid \Hom_{\mathbf{D}(A)}(\Sigma,X[i]) = 0 ~\forall i>0\}.$$

Suppose that $\Sigma$ is a projective $n$-presentation of $T$, so that $T=H^0(\Sigma)$. Recall from \cite[Definition 3.1]{mao22} that $T$ is called an \emph{$n$-silting module} (with respect to $\Sigma$) if $\mathrm{Pres}^{n-1}(T)=\mathcal{D}_{\Sigma}$.
In this case, the class $\mathrm{Pres}^{n-1}(T)$ is called an \emph{$n$-silting class}.

Let $\Omega$ be the object $E_{0} \rightarrow\cdots \rightarrow E_{n-1}\rightarrow E_{n}$ in $\CopresE{n}$. Let
\begin{center}
$\mathcal{B}_{\Omega}=\left\{X \in \mathrm{Mod}(A)~~ \bigg| ~~{\left.
\begin{aligned}
&\text{\rm The induced complex } \Hom_{A}( X, E_{0})\rightarrow \cdots\rightarrow\Hom_{A}( X, E_{n-2}) \\
&\rightarrow \Hom_{A}(X, E_{n-1})\rightarrow 0\text{ is exact} \end{aligned}
\right. }
                                                               \right\}.$
\end{center}

In other words, observe that we have 
$$\mathcal{B}_{\Omega} = \{X \in \mathrm{Mod}(A) \mid \Hom_{\mathbf{D}(A)}(X,\Omega[i]) = 0 ~\forall i>0\}.$$

Suppose that $\Omega$ is an injective $n$-copresentation of $C$, so that $C = H^0(\Omega)$.
Recall from \cite[Definition 4.1]{mao22} that $C$ is called an \emph{$n$-cosilting module} (with respect to $\Omega$) if $\mathrm{Copres}^{n-1}(C)= \mathcal{B}_{\Omega}$.
In this case, the class $\mathrm{Copres}^{n-1}(C)$ is called an \emph{$n$-cosilting class}.

\begin{remark} $(1)$  A $1$-silting module is just the silting module
in the sense of \cite{li16}.
A 1-cosilting module is just the cosilting module in the sense of \cite{br17}. From \cite[Theorem 4.9]{li16}, we know that  for an arbitrary ring there is a
bijection between (not necessarily finitely generated) silting modules and
(not necessarily compact) 2-silting complexes.

$(2)$ Let $T$ be an $A$-module. By \cite[Propositions 3.3 and 4.3]{mao22},  $T$ is $n$-tilting if and only if $T$ is a $n$-silting $A$-module with respect to a projective n-presentation $ P_{n}\s{\sigma}\rightarrow P_{n-1}\rightarrow \cdots \rightarrow P_{0}$ with $\sigma$ monic. In addition, $C$ is $n$-cotilting if and only if $C$ is a $n$-cosilting
module with respect to an injective $n$-copresentation
$E_{0} \rightarrow\cdots \rightarrow E_{n-1}\s{\omega}\rightarrow E_{n}$
with $\omega$ epic.
\end{remark}

Following \cite{mao22}, we say that an $n$-silting module $T$ is \textit{of finite type} if there is a subset $\mathcal{S}$ of $\PresFP{n}$ such that $\mathcal{D}_\Sigma = \bigcap_{\sigma \in \mathcal{S}}\mathcal{D}_\sigma$. Any $n$-tilting module and any $1$-silting module is of finite type \cite{sto07,ma19}. However, we shall demonstrate in \S\ref{counterexample} that a $2$-silting module may fail to be of finite type.

\section{  $n$-silting modules and silting complexes }

The following lemma is inspired by \cite[Lemma 5.4]{ma19}.
\begin{lemma} \label{key lem}
Let $\Sigma$ and $\Omega$ be the objects in $\PresP{n}$ and  $\CopresE{n}$, respectively.
Then for any $X \in \Pres{n}{}$ and $Y \in \Copres{n}{}$ we have:
 \begin{itemize}
 \item [$(i)$] $X$ belongs to $\Sigma^{\perp_{>0}}$ if and only if
 $H^0(X) \in \mathcal{D}_{\Sigma}$;
 \item [$(ii)$] $Y$ belongs to $^{\perp_{>0}}\Omega$ if and only if
 $ H^0(Y) \in \mathcal{B}_{\Omega}$.
 \end{itemize}
\end{lemma}

\begin{proof} We only prove $(1)$ and the proof of $(2)$ is similar. Since $X \in \Pres{n}{}$, there is the soft truncation triangle of $X$ of the form
  $$N[n] \to X \to M[0] \xrightarrow{+},$$
where $M = H^0(X)$ and $N=H^{-n}(X)$. Applying $\Hom_{\mathbf{D}(A)}(\Sigma,-)$ yields for any $i>0$ an exact sequence of the form
  $$\Hom_{\mathbf{D}(A)}(\Sigma,N[n+i]) \to \Hom_{\mathbf{D}(A)}(\Sigma,X[i]]) \to \Hom_{\mathbf{D}(A)}(\Sigma,M[i]) \to \Hom_{\mathbf{D}(A)}(\Sigma,N[n+i+1]).$$
Since $\Sigma \in \PresP{n}$, both the leftmost and the rightmost terms vanish, and thus we obtain for any $i>0$ an isomorphism 
  $$\Hom_{\mathbf{D}(A)}(\Sigma,X[i]) \cong \Hom_{\mathbf{D}(A)}(\Sigma,M[i]).$$
It remains to note that the vanishing of the right-hand term above for all $i>0$ is equivalent to $M \in \mathcal{D}_\Sigma$, while the vanishing of the left-hand term for all $i>0$ is equivalent to $X \in \Sigma^{\perp_{>0}}$.
\end{proof}

\begin{lemma}\label{res}
  Let $\Sigma \in \PresP{n}$ and $\Omega \in \CopresE{n}$. Then we have $\mathrm{Pres}^{n-1}(\mathcal{D}_\Sigma) \subseteq \mathcal{D}_\Sigma$ and $\mathrm{Copres}^{n-1}(\mathcal{B}_\Omega) \subseteq \mathcal{B}_\Omega$.
\end{lemma}
\begin{proof}
  We prove only the claim about $\mathcal{D}_\Sigma$, the other follows by a dual argument. For any $A$-module $M$, consider a short exact sequence
  $$0 \to K \to D \to M \to 0$$
  with $D \in \mathcal{D}_\Sigma$. Applying $\operatorname{Hom}_{\mathbf{D}(A)}(\Sigma,-)$ we obtain for any $i>0$ an exact sequence of the form
  $$\operatorname{Hom}_{\mathbf{D}(A)}(\Sigma,D[i]) \to \operatorname{Hom}_{\mathbf{D}(A)}(\Sigma,M[i]) \to \operatorname{Hom}_{\mathbf{D}(A)}(\Sigma,N[i+1]) \to \operatorname{Hom}_{\mathbf{D}(A)}(\Sigma,D[i+1]).$$
  As $D \in \mathcal{D}_\Sigma$, we see that the left-most and the right-most group in the last display vanish, and therefore we obtain an isomorphism $\operatorname{Hom}_{\mathbf{D}(A)}(\Sigma,M[i]) \cong \operatorname{Hom}_{\mathbf{D}(A)}(\Sigma,N[i+1])$ for any $i>0$. As $\Sigma \in \PresP{n}$, this shows that $\operatorname{Hom}_{\mathbf{D}(A)}(\Sigma,M[i]) = 0$ for any $i \geq n$.

  Let now $M \in \mathrm{Pres}^{n-1}(\mathcal{D}_\Sigma)$. Using the previous argument inductively, we see $\operatorname{Hom}_{\mathbf{D}(A)}(\Sigma,M[i])$ vanishes for any $i>0$, which amounts precisely to the desired $M \in \mathcal{D}_\Sigma$.
\end{proof}

\begin{lemma}\label{one-implication}
  \begin{itemize}
    \item [$(i)$] Let $\Sigma \in \PresP{n}$. If $\Sigma$ is an $(n+1)$-term silting complex then $H^0(\Sigma)$ is an $n$-silting module with respect to $\Sigma$.
    \item [$(ii)$] Dually, let $\Omega \in \CopresE{n}$. If $\Omega$ is an $(n+1)$-term cosilting complex then $H^0(\Omega)$ is an $n$-cosilting module with respect to $\Omega$.
   \end{itemize}
\end{lemma}
\begin{proof}
  \textit{$(i)$} Set $T=H^0(\Sigma)$. By Lemma~\ref{key lem}, we have $\mathrm{Add}(\Sigma) \subseteq \Sigma^{\perp_{>0}}$ if and only if $\mathrm{Add}(T) \subseteq \mathcal{D}_\Sigma$. Then it follows from Lemma~\ref{res} that $\mathrm{Pres}^{n-1}(T) \subseteq \mathcal{D}_\Sigma$. Assume now that $\Sigma$ is a silting complex. By \cite[Lemma 3.11]{wei13}, we have for any $N \in \Sigma^{\perp_{>0}}$ a collection of triangles indexed by $i\geq 0$ of the form
   $$N_{i+1} \to \Sigma_i \to N_i \to N_{i+1}[1],$$
  where $N_0 = N$, $T_i \in \mathrm{Add}(\Sigma)$, and $N_i \in \Sigma^{\perp_{>0}}$ for all $i \geq 0$. Consider $M \in \mathcal{D}_\Sigma$ and let $N \in \PresP{n}$ be any projective $n$-presentation of $M$. Then $N \in \Sigma^{\perp_{>0}}$ by Lemma~\ref{key lem} and the above observation applies. Applying $H^0$ to the triangles yields for any $i \geq 0$ an exact sequence of the form
  $$M_{i+1} \to T_i \to M_i \to 0,$$
  where $M_i = H^0(N_i)$ and $T_i = H^0(\Sigma_i)$, note that $H^{-1}(N_{i}) = 0$ for any $i \geq 0$ as $N_i \in \Sigma^{\perp_{>0}} \subseteq \mathbf{D}^{\leq 0}$. Since $T_i \in \mathrm{Add}(T)$, we conclude that $M \in \mathrm{Pres}^{n-1}(T)$ as desired.

  \textit{$(ii)$} Dual.  
\end{proof}

Towards proving a partial converse to Lemma~\ref{one-implication}, the following condition will be useful to consider for an object $Z \in \mathrm{Mod}_{\mathcal{E}}(AQ_{n+1}/I)$:

$(\dag)$ The cotorsion pair $\mathfrak{C}=({}^{\perp_{\infty
}}(Z^{\perp_{\infty}}), Z^{\perp_{\infty}})$ in $\mathrm{Mod}(AQ_{n+1}/I)$ associated to $Z$ has its kernel $\Ker(\mathfrak{C})={^{\perp_{\infty}}(Z^{\perp_{\infty}})\cap Z^{\perp_{\infty}}}$ contained in $\mathrm{Mod}_{\mathcal{E}}(AQ_{n+1}/I)$.

Note that if $T$ is a tilting module in $\mathrm{Mod}_{\mathcal{E}}(AQ_{n+1}/I)$ then it surely satisfies $(\dag)$ as then $\Ker(\mathfrak{C}) = \mathrm{Add}(T)$.

\begin{lemma} \label{lem tilting} Let $\Sigma$ be any object in $\mathrm{Mod}_{\mathcal{P,E}}(AQ_{n+1}/I)$. Suppose that the condition $(\dag)$ holds for $\Sigma$ and that ${\Sigma}^{\perp_{\infty}}\cap\mathrm{Mod}_{\mathcal{E}}(AQ_{n+1}/I)$ is closed under direct sums in $\mathrm{Mod}(AQ_{n+1}/I)$.
Then ${\Sigma}^{\perp_{\infty}}$ is closed under direct sums in $\mathrm{Mod}(AQ_{n+1}/I)$.
\end{lemma}
\begin{proof} 
We know from \cite[Lemma 2.5(6)]{ma18} that the projective dimension of $\Sigma$ is at most $n$ in $\mathrm{Mod}(AQ_{n+1}/I)$.
By \cite[\S 6]{tr12}, there exists a cotorsion pair $(^{\perp_{\infty}}(\Sigma^{\perp_{\infty}}),\Sigma^{\perp_{\infty}})$ in $\mathrm{Mod}(AQ_{n+1}/I)$ induced by $\Sigma$.
Using condition $(\dag)$, we obtain that the class
\[\mathcal{X}:={^{\perp_{\infty}}(\Sigma^{\perp_{\infty}})}\cap {\Sigma}^{\perp_{\infty}}={^{\perp_{\infty}}(\Sigma^{\perp_{\infty}})}\cap {\Sigma}^{\perp_{\infty}}\cap\mathrm{Mod}_{\mathcal{E}}(AQ_{n+1}/I)\] is closed under direct sums in $\mathrm{Mod}(AQ_{n+1}/I)$.
Now we claim that the class ${\Sigma}^{\perp_{\infty}}$ equals to the class of all $\mathcal{X}$-resolved modules.
Indeed, if $M\in{\Sigma}^{\perp_{\infty}}$, then an $\mathcal{X}$-resolution is obtained by an iteration of special
${^{\perp_{1}}(\Sigma^{\perp_{\infty}})}$-precovers.
Conversely, assume there exists an $\mathcal{X}$-resolution of an object $B$
$$\cdots\rightarrow E_{n}\rightarrow \cdots \rightarrow E_{0}\rightarrow B\rightarrow 0$$
Denote by $K_{0}$ the kernel of the epimorphism $E_{0}\rightarrow B$,
by $K_{1}$ the kernel of the epimorphism $E_{1}\rightarrow K_{0}$, etc.
Let $X\in {^{\perp_{1}}(\Sigma^{\perp_{\infty}})}$. Then
$\Ext^{1}_{AQ_{n+1}/I}(X,B)\cong \Ext^{2}_{AQ_{n+1}/I}(X,K_{0})\cong \cdots \cong\Ext^{n+1}_{AQ_{n+1}/I}(X,K_{n-1})=0$ as $X$ has projective dimension at most $n$, so $B\in {\Sigma}^{\perp_{\infty}}$.
This proves the claim, which in turn implies that ${\Sigma}^{\perp_{\infty}}$ is closed under direct sums.
\end{proof}

Given $-n \leq j < 0$, let $\overline{A}_{j} \in \Mod(AQ_{n+1}/I)$ be the ``disk'' representation which corresponds under the functor $\Theta: \Mod(AQ_{n+1}/I) \to \Cpx{A}{-n}{0}{}$ to the complex $(\cdots \to 0 \to 0 \to  A \xrightarrow{=} A \to 0 \to 0 \to \cdots)$ concentrated in degrees $j$ and $j+1$.

\begin{lemma}\label{equiv-conditions}
  Let $\Sigma \in \PresP{n}$ be a projective $n$-presentation of an $n$-silting module $T = H^0(\Sigma)$. The following conditions are equivalent:
  \begin{enumerate}
    \item[$(i)$] $\Sigma$ is a silting complex in $\mathbf{D}(A)$,
    \item[$(ii)$] $\Sigma$ satisfies the condition $(\dag)$,
    \item[$(iii)$] $\Sigma^{\perp_{\mathbb{Z}}} = 0$.
  \end{enumerate}
\end{lemma}
\begin{proof}
  $(i) \implies (ii):$ Since $\Sigma$ is a silting complex in $\mathbf{D}(A)$, it also represents a tilting module in $\Mod_\mathcal{P}(AQ_{n+1}/I)$ by Theorem~\ref{mv-bijection}. Then $\Sigma$ clearly satisfies $(\dag)$.

  $(ii) \implies (iii):$ 
By Lemma \ref{key lem}, the fact that $\mathcal{D}_{\Sigma}$ is closed under coproducts in $\mathrm{Mod}(A)$ corresponds to ${\Sigma}^{\perp_{\infty}}\cap \mathrm{Mod}_{\mathcal{E}}(AQ_{n+1}/I)$ being closed under direct sums. Hence, by Lemma \ref{lem tilting}, we may conclude that $\Sigma^{\perp_\infty}$ is closed under direct sums in $\mathrm{Mod}(AQ_{n+1}/I)$ (which together with $\Sigma \in \Sigma^{\perp_\infty}$ amounts to $\Sigma$ being a \textit{partial $n$-tilting module} in $\mathrm{Mod}(AQ_{n+1}/I)$).

Let $E=\bigoplus^{n}_{i=1}\overline{A}_{j}$.
Since $E$ is projective and injective in
$\mathrm{Mod}_{\mathcal{P}}(AQ_{n+1}/I)$ by \cite[Lemma 2.5]{ma18}, the object
$Z :={\Sigma}\oplus E$ is also partial $n$-tilting in $\mathrm{Mod}(AQ_{n+1}/I)$ and $Z^{\perp_\infty} = \Sigma^{\perp_\infty}$. Since $E$ is contractible when viewed as a cochain complex, we have $\Phi(Z) \cong \Phi(\Sigma)$ in $\mathbf{D}(A)$, and thus it is sufficient to establish the condition $(iii)$ for $Z$ instead of for $\Sigma$.

We first claim that $Z^{\perp_{\infty}}\cap\mathrm{Mod}_{\mathcal{E}}(AQ_{n+1}/I)\subseteq \mathrm{Gen}(Z)$ in $\mathrm{Mod}(AQ_{n+1}/I)$.
Indeed, write $\Sigma$ as
$$P_{n}\s{\sigma_{n}}\rightarrow P_{n-1}\rightarrow\cdots \rightarrow P_{1}\s{\sigma_{1}} \rightarrow P_0$$
and let $X\in\mathrm{Mod}_{\mathcal{E}}(AQ_{n+1}/I)$ be of the form
$$X_{n}\s{d_{n}}\rightarrow X_{n-1}\rightarrow\cdots \rightarrow X_{1}\s{d_{1}}\rightarrow X_{0}$$
such that $X\in Z^{\perp_{\infty}}$, that is, $\mathrm{Coker}(d_{1})\in \mathcal{D}_{\Sigma}=\mathrm{Pres}^{n-1}(T)$.
Hence, there is a surjection $p : T^{(I_{0})}\rightarrow \mathrm{Coker}(d_{1})$.
Since ${\Sigma}^{(I_{0})}\in \mathrm{Mod}_{\mathcal{P},\mathcal{E}}(AQ_{n+1}/I)$ and $X\in\mathrm{Mod}_{\mathcal{P},\mathcal{E}}(AQ_{n+1}/I)$, we can lift $p_{0}$ to a morphism ${\Sigma}^{(I_{0})}\rightarrow X$:
$$\xymatrix{
  P^{(I_{0})}_{n} \ar[d]_{} \ar[r]^{\sigma^{(I_{0})}_{n}} & P^{(I_{0})}_{n-1} \ar[d]_{} \ar[r]^{} & \cdots \ar[r]^{}& P^{(I_{0})}_{1} \ar[d]_{} \ar[r]^{\sigma^{(I_{0})}_{1}}  & P^{(I_{0})}_{0} \ar[d]_{} \ar[r]^{\sigma^{(I_{1})}_{0}} & T^{(I_{0})} \ar[d]_{p} \ar[r]^{} & 0  \\
  X_{n} \ar[r]^{d_{n}} & X_{n-1} \ar[r]^{d_{n-1}} & \cdots \ar[r]^{} & X_{1} \ar[r]^{d_{1}}& X_{0} \ar[r]^-{d_{0}} & \mathrm{Coker}(d_{1}) \ar[r]^{} & 0.   }
$$
Furthermore, take a surjection
$q_{i} : A^{(J_{i})}\rightarrow X_{i}$ for each $i = -n,-n+1,\ldots,-1$ and jointly extend them to a morphism
$E^{(J_{0})}\rightarrow X$ for a suitable cardinal $J_{0}$.
One can check that the summed morphism ${\Sigma}^{(I_{0})}\oplus E^{(J_{0})}\rightarrow X$ is surjective.
It follows that $X\in \mathrm{Gen}(Z)$.

Next, consider the complete cotorsion pair $({^{\perp_{1}}(Z^{\perp_{\infty}})}, Z^{\perp_{\infty}})$ induced by $Z$.
We claim that ${^{\perp_{1}}(Z^{\perp_{\infty}})}\cap Z^{\perp_{\infty}}=\mathrm{Add}(Z)$.
We already know that $\mathrm{Add}(Z)\subseteq {^{\perp_{1}}(Z^{\perp_{\infty}})}\cap Z^{\perp_{\infty}}$.
Conversely, let $M\in {^{\perp_{1}}(Z^{\perp_{\infty}})}\cap Z^{\perp_{\infty}}$.
Consider an $\mathrm{Add}(Z)$-precover $g:Y\rightarrow M$ of $M$, where $Y\in \mathrm{Add}(Z)$, that is, for any homomorphism $f:F\rightarrow M$ where 
$F$ is in $\mathrm{Add}(Z)$, there exists a homomorphism $h:F\rightarrow Y$ such that $f=gh$.
By condition $(\dag)$, we deduce that $M\in \mathrm{Mod}_{\mathcal{E}}(AQ_{n+1}/I)$.
Therefore, $g$ is an epimorphism because
$M\in Z^{\perp_{\infty}}\cap\mathrm{Mod}_{\mathcal{E}}(AQ_{n+1}/I)\subseteq \mathrm{Gen}(Z)$.
Then we obtain an exact sequence
$0 \rightarrow K \rightarrow Y\rightarrow  M \rightarrow 0$
with $K\in T^{\perp_{\infty}}$ by $(T2)$.
It follows that the sequence splits, and $M\in \mathrm{Add}(Z)$.
The claim holds.

Let $P$ be any projective module in $\mathrm{Mod}(AQ_{n+1}/I)$.
By the completeness of the cotorsion pair $({^{\perp_{1}}(Z^{\perp_{\infty}})}, Z^{\perp_{\infty}})$, there is a special $Z^{\perp_{\infty}}$-preenvelope,
$\psi: P\rightarrow  Z_{0}$, of $P$ with $C_{1} :=\mathrm{Coker}(\psi)\in {^{\perp_{1}}(Z^{\perp_{\infty}})}$.
Since $P\in {^{\perp_{1}}(Z^{\perp_{\infty}})}$, also $Z_{0}\in {^{\perp_{1}}(Z^{\perp_{\infty}})}$.
By the claim above, we see that $Z_{0}\in \mathrm{Add}(Z)$.
Iterating the above construction, we get an exact sequence
$$0\rightarrow P\rightarrow Z_{0}\rightarrow Z_{1}\rightarrow\cdots\rightarrow Z_{n-1}\rightarrow C_{n}\rightarrow 0.$$
with $Z_{i}\in \mathrm{Add}(Z)$ for $0 \leq i\leq n-1$ and $C_{n}$ in ${^{\perp_{1}}(Z^{\perp_{\infty}})}$.
By dimension shifting, we have
$\Ext^{i}_{AQ_{n+1}/I}(Z,C_{n})\cong \Ext^{i+n}_{AQ_{n+1}/I}(Z,P)$
for each $i>0$, hence $C_{n}\in {^{\perp_{1}}(Z^{\perp_{\infty}})}\cap Z^{\perp_{\infty}}=\mathrm{Add}(Z)$.

Finally, let $P$ be the $\mathrm{Mod}(AQ_{n+1}/I)$-module corresponding to the stalk cochain complex $(\cdots \to 0 \to 0 \to A \to 0 \to 0 \to \cdots)$ concentrated in degree $0$. Then $P$ is projective as an $\mathrm{Mod}(AQ_{n+1}/I)$-module by \cite[Lemma 2.4]{ma18}, and thus the preceding paragraph applies to it, showing that $A$ is in the smallest localizing subcategory of $\mathbf{D}(A)$ generated by $Z$. As $A^{\perp_\mathbb{Z}} = 0$, the condition $(iii)$ follows.

  $(iii) \implies (i):$ By Lemma~\ref{key lem}, we know that $\mathrm{Add}(\Sigma) \subseteq \Sigma^{\perp_{>0}}$. Under condition $(iii)$, $\Sigma$ is an $(n+1)$-term silting complex by Lemma~\ref{silting-complex-generate}.
\end{proof}

On the cosilting side we formulate a simpler criterion.

\begin{lemma}\label{equiv-conditions-co}
  Let $\Omega \in \CopresE{n}$ be an injective $n$-copresentation of an $n$-cosilting module $C = H^0(\Sigma)$. The following conditions are equivalent:
  \begin{enumerate}
    \item[(i)] $\Omega$ is a cosilting complex in $\mathbf{D}(A)$,
    \item[(ii)] ${}^{\perp_{\mathbb{Z}}}\Omega = 0$.
  \end{enumerate}
\end{lemma}
\begin{proof}
  $(i) \implies (ii):$ This is Lemma~\ref{silting-complex-generate}.
  
  $(ii) \implies (i):$ We already know that $\mathrm{Prod}(\Omega) \subseteq {}^{\perp_{>0}}\Omega$ by Lemma~\ref{key lem}, the rest follows from Lemma~\ref{silting-complex-generate}.
\end{proof}

\begin{lemma}\label{gen-cogen-module}
  Let $T$ be an $n$-silting module with respect to a projective $n$-presentation $\Sigma \in \PresP{n}$. Then $\Sigma^{\perp_{\mathbb{Z}}}$ contains no non-zero $A$-module.

  Let $C$ be an $n$-cosilting module with respect to an injective $n$-copresentation $\Omega \in \CopresE{n}$. Then ${}^{\perp_{\mathbb{Z}}}\Sigma$ contains no non-zero $A$-module.
\end{lemma}
\begin{proof}
   Let $M \in \Sigma^{\perp_{\mathbb{Z}}}$ be an $A$-module $M$. First, as $\Sigma^{\perp_{\mathbb{Z}}} \subseteq \Sigma^{\perp_{>0}}$, we have $M \in \mathcal{D}_\Sigma$. But since $0 = \Hom_{\mathbf{D}(A)}(\Sigma,M) \cong \Hom_{A}(T,M)$, we see quickly that $M \in \mathrm{Pres}^{n-1}(T) \subseteq \mathrm{Gen}(T)$ implies the desired $M=0$. The other claim follows dually.
\end{proof}

In general, there seems to be no reason why $\Sigma^{\perp_{\mathbb{Z}}}$ not containing any non-zero $A$-module, even in presence of the condition $\mathrm{Add}(\Sigma) \subseteq \Sigma^{\perp_{>0}}$, should imply that $\Sigma$ is a silting complex; similarly for the dual cosilting claim. However, one can see that this is sufficient for particular classes of rings. Recall that $\Sigma^{\perp_{\mathbb{Z}}}$ is always a \textit{colocalizing subcategory} of $\mathbf{D}(A)$, that is, a full triangulated subcategory closed under all products, and that ${}^{\perp_{\mathbb{Z}}} \Omega$ is a \textit{localizing subcategory} of $\mathbf{D}(A)$, that is, a full triangulated subcategory closed under all coproducts. 

Consider the following condition:

$(\ddag)$ Any non-zero localizing and any non-zero colocalizing subcategory of $\mathbf{D}(A)$ contains a non-zero $A$-module.

We are ready to state our two main Theorems whose proofs follow by a straightforward combination of the previous results of this section.

\begin{theorem}\label{thm1}
Let $T$ be an $A$-module and $\Sigma$ be a projective $n$-presentation of $T$.
There is an injective map
      $$\left \{
      \begin{tabular}{ccc} \text{$n$-tilting modules in $\mathrm{Mod}(AQ_{n+1}/I)$} \\ \text{which belong to $\mathrm{Mod}_{\mathcal{P},\mathcal{E}}(AQ_{n+1}/I)$} \\ \text{up to equivalence} \end{tabular}\right \}
      \hookrightarrow
      \left \{\begin{tabular}{ccc} \text{$n$-silting $A$-modules} \\ \text{up to equivalence} \end{tabular}
      \right \}$$
     induced by the assignment
       $$\Sigma \mapsto T.$$
Assume that $A$ satisfies the condition $(\ddag)$, or the condition $(\dag)$ holds for the projective $n$-presentation $\Sigma$ associated to any $n$-silting $A$-module. Then this assignment is a one-to-one correspondence.
\end{theorem}

\begin{theorem}\label{thm1co}
Let $C$ be an $A$-module and $\Omega$ be an injective $n$-copresentation of $C$.
There is an injective map
      $$\left \{
      \begin{tabular}{ccc} \text{$n$-cotilting modules in $\mathrm{Mod}(AQ_{n+1}/I)$} \\ \text{which belong to $\mathrm{Mod}_{\mathcal{P},\mathcal{I}}(AQ_{n+1}/I)$} \\ \text{up to equivalence} \end{tabular}\right \}
      \hookrightarrow
      \left \{\begin{tabular}{ccc} \text{$n$-cosilting $A$-modules} \\ \text{up to equivalence} \end{tabular}
      \right \}$$
     induced by the assignment
       $$\Omega \mapsto C.$$
Assume that $A$ satisfies the condition $(\ddag)$, then this assignment is a one-to-one correspondence.
\end{theorem}

The condition $(\ddag)$ is satisfied in many settings of interest. In fact, it seems non-trivial to find an explicit example of a ring $A$ for which $(\ddag)$ fails.

\begin{lemma}\label{dagger-holds}
  The condition $(\ddag)$ holds if $A$ is any of the following:
  \begin{enumerate}
    \item left hereditary,
    \item commutative noetherian.
  \end{enumerate}
\end{lemma}
\begin{proof}
    \textit{(1)} If $A$ is left hereditary then any object of $\mathbf{D}(A)$ is isomorphic to a direct sum of shifts of its cohomology modules, see e.g. \cite[\S 1.6]{kra07}, from which the first claim immediatelly follows.

    \textit{(2)} If $A$ is commutative noetherian then any localizing (resp. colocalizing) subcategory is generated (resp. cogenerated) by residue field modules, see \cite{nee11}. 
\end{proof}

\begin{definition}
  Let us say that an $(n+1)$-term silting complex $\Sigma$ \textit{represents} an $n$-silting module if $\Sigma \in \PresP{n}$. By Lemma~\ref{one-implication}, this amounts precisely to the fact that $H^0(\Sigma)$ is an $n$-silting module and $\Sigma$ is isomorphic to a projective $n$-presentation witnessing this fact. Dually, we speak about $(n+1)$-term cosilting complex which \textit{represent} an $n$-cosilting module.
\end{definition}

We recall that a short exact sequence in $\mathrm{Mod}(A)$ is said to be \emph{pure-exact}
if the covariant functor $\Hom_{A}(F,-)$ preserves its exactness for every finitely presented module $F$.
An $A$-module $U$ is \emph{pure-injective} if the contravariant functor
$\Hom_{A}(-,U)$ preserves the exactness of every pure-exact sequence.
In the following, we prove that all $n$-cosilting modules are pure-injective.
\begin{corollary}\label{pure-injective} 
Assume that $A$ satisfies the condition $(\ddag)$. If $C$ is an $n$-cosilting $A$-module with respect to $\Omega$,
then $C$ is pure-injective.
\end{corollary}
\begin{proof}
By Theorem \ref{thm1}, we know that $\Omega$ is a cosilting complex and
so is pure-injective as an object of $\mathbf{D}(A)$ by \cite[Proposition 3.10]{ma18}. Then $C=H^0(\Omega)$ is pure-injective by \cite[Theorem 17.3.19]{prest09}.
\end{proof}

\section{The case of commutative noetherian rings}
\subsection{Commutative Algebra setup} In this section, $A$ will always denote a commutative noetherian ring with Zariski spectrum $\operatorname{Spec}(A)$. Let us set up some notation, where $\pp \in \operatorname{Spec}(A)$ always denotes a prime ideal of $A$:
\begin{itemize}
  \item Let $Y_\mathfrak{p}=Y\oo_{A} A_\mathfrak{p}$ denote the localization of an object $Y \in \mathbf{D}(A)$ at $\mathfrak{p}$.
  \item The support of an $A$-module $M$ is defined as $\operatorname{Supp}(M) = \{\mathfrak{p} \in \operatorname{Spec}(A) \mid M_\mathfrak{p} \neq 0\}$.
  \item We let $V(\mathfrak{p}) = \{\mathfrak{q} \in \operatorname{Spec}(A) \mid \mathfrak{p} \subseteq \mathfrak{q}\}$ denote the Zariski closure of $\mathfrak{p}$ in $\operatorname{Spec}(A)$.
  \item The subcategory of all modules whose support is a subset of $V(\mathfrak{p})$ forms a localizing subcategory of $\mathrm{Mod}(A)$ with the induced torsion radical functor $\Gamma_{V(\mathfrak{p})}: \mathrm{Mod}(A) \to \mathrm{Mod}(A)$.
  \item The right derived functor $\mathbf{R}\Gamma_{V(\mathfrak{p})}(-): \mathbf{D}(A) \to \mathbf{D}(A)$ is called the \textit{total local cohomology} functor at $V(\mathfrak{p})$.
  \item For any $Y \in \mathbf{D}(A)$, we let
    $$H_\mathfrak{p}^i(Y) := H^i(\mathbf{R}\Gamma_{V(\mathfrak{p})}(Y))$$
  denote the $i$-th \textit{local cohomology} module at $\mathfrak{p}$.
  \item The $\mathfrak{p}$-local \textit{depth} of $Y$ is then defined as
  $$\operatorname{depth}_\mathfrak{p}(Y)=\inf \{i \in \mathbb{Z} \mid H^i_\mathfrak{p}(Y_\mathfrak{p}) \neq 0\}.$$
  \item In particular, recall that if $(A,\mathfrak{m},k)$ is a local ring with maximal ideal $\mathfrak{m}$ and residue field $k$, its \textit{depth} is defined as $\operatorname{depth}(A) = \operatorname{depth}_\mathfrak{m}(A)$.
\end{itemize}
In the Commutative Algebra setting, we have access to a full classification of silting and cosilting complexes. A \textit{monotone perversity} is a function $f: \operatorname{Spec}(A) \to \mathbb{Z}$ such that $f(\mathfrak{p}) \leq f(\mathfrak{q})$ whenever $\mathfrak{p}, \mathfrak{q} \in \operatorname{Spec}(A)$ are two prime ideals such that $\mathfrak{p} \subseteq \mathfrak{q}$. We say that a monotone perversity $\operatorname{Spec}(A)$ is \textit{bounded} if there is $N \in \mathbb{Z}$ such that $f(\mathfrak{p}) \leq N$ for all $\mathfrak{p} \in \operatorname{Spec}(A)$. Note that any monotone perversity is bounded if $(A,\mathfrak{m},k)$ is a local ring, as then we can choose $N = f(\mathfrak{m})$.

\begin{theorem}\label{classification-noeth}
  Let $A$ be a commutative noetherian ring. Then there is a bijection
  $$\left \{ \begin{tabular}{ccc} \text{bounded monotone} \\ \text{perversities on $\operatorname{Spec}(A)$} \end{tabular}\right \}  \xleftrightarrow{1-1}  \left \{ \begin{tabular}{ccc} \text{cosilting complexes} \\ \text{in $\mathbf{D}(A)$} \\ \text{up to equivalence} \end{tabular}\right \}.$$
  The bijection sends a monotone perversity $f$ to a cosilting t-structure $(\mathcal{U}_f,\mathcal{V}_f) = ({}^{\perp_{>0}}\Omega,{}^{\perp_{\leq 0}}\Omega)$ of the form
  $$\mathcal{U}_f = \{X \in \mathbf{D}(A) \mid \operatorname{Supp}(H^i(X)) \subseteq \{\mathfrak{p} \in \operatorname{Spec}(A) \mid f(\mathfrak{p})>i\} ~\forall i \in \mathbb{Z}\},$$
  $$\mathcal{V}_f = \{X \in \mathbf{D}(A) \mid \operatorname{depth}_{\mathfrak{p}}(X)\geq f(\mathfrak{p}) ~\forall \mathfrak{p} \in \operatorname{Spec}(A)\}.$$
  Moreover, any cosilting complex is of cofinite type, and thus the character duality 
  $$\Sigma \mapsto \Omega = \Sigma^+$$ 
  induces a bijection between equivalence classes of silting and cosilting complexes over $A$.

  Furthermore, under this correspondence:
  \begin{itemize}
    \item The $(n+1)$-term (co)silting complexes correspond precisely to those monotone perversities $f: \operatorname{Spec}(A) \to \mathbb{Z}$ such that $0 \leq f(\mathfrak{p}) \leq n$ for any $\mathfrak{p} \in \operatorname{Spec}(A)$.
    \item The $n$-tilting modules correspond precisely to those monotone perversities $f: \operatorname{Spec}(A) \to \mathbb{Z}$ such that $0 \leq f(\mathfrak{p}) \leq \operatorname{depth}(A_\mathfrak{p})$ for any $\mathfrak{p} \in \operatorname{Spec}(A)$.
  \end{itemize}

\end{theorem}
\begin{proof}
  The main bijection follows from the classification result {\cite[Theorem 3.8]{ahh21}}, see also the exposition in \cite[\S 3]{v23} and \cite[\S 3]{hns24}. In particular, the statement about cofinite type was proved in \cite{hn21}. The first bullet claim follows from the proof of \cite[Theorem 3.8]{ahh21}, while the second is discussed in \cite[Remark 5.11]{hns24}.
\end{proof}
The following provides a natural generalization of \cite[Corollary 5.12]{pv21}, which constitutes the $n=1$ case.
\begin{theorem}\label{cotilting-complex}
  Let $\Omega$ be a cosilting complex which represents an $n$-cosilting module. Then $\Omega$ is a cotilting complex.
\end{theorem}
\begin{proof}
  The goal is to show that $\Omega \in {}^{\perp_{< 0}}\Omega$ where, by Theorem~\ref{classification-noeth}, there is a monotone perversity $f: \operatorname{Spec}(A) \to \mathbb{Z}$ such that
  $${}^{\perp_{< 0}}\Omega= \mathcal{U}_f[-1] = \{X \in \mathbf{D}(A) \mid \operatorname{Supp}(H^i(X)) \subseteq \{\mathfrak{p} \in \operatorname{Spec}(A) \mid f(\mathfrak{p})\geq i\} ~\forall i \in \mathbb{Z}\}.$$
  For each $-1 \leq i <n$, let $V_i = \{\mathfrak{p} \in \operatorname{Spec}(A) \mid f(\mathfrak{p})>i\}$, a subset of $\operatorname{Spec}(A)$ closed under specialization, note that $V_{-1} = \operatorname{Spec}(A)$. We then need to show
  $$\operatorname{Supp}(H^i(\Omega)) \subseteq V_{i-1}$$
  for each $0<i\leq n$. Since $\Omega$ represents an $n$-cosilting module, we have $H^i(\Omega) = 0$ for $0<i<n$, and thus the only thing required is to show $\operatorname{Supp}(K) \subseteq V_{n-1}$, where $K = H^n(\Omega)$. Let us write $\Omega$ as a complex $E_0 \to E_1 \to \cdots E_n$ of injective $A$-modules in degrees $0,\ldots,n$, so that we have an exact sequence $E_{n-1} \xrightarrow{\varphi} E_n \to K \to 0$. Because $A$ is commutative noetherian, the injective $A$-module $E_n$ splits as $E_n \cong E \oplus E'$ in a way that $\operatorname{Supp}(E) \subseteq V_{n-1}$ and $\operatorname{Ass}(E') \cap V_{n-1} = \emptyset$. Since $\operatorname{depth}_\mathfrak{q}(E') = \infty$ for any $\mathfrak{q} \in V_{n-1}$, we clearly have $E'[-n+1] \in {}^{\perp_{>0}}\Omega$, where again by Theorem~\ref{classification-noeth}
  $${}^{\perp_{>0}}\Omega = \mathcal{V}_f = \{X \in \mathbf{D}(A) \mid \operatorname{depth}_{\mathfrak{p}}(X)\geq f(\mathfrak{p}) ~\forall \mathfrak{p} \in \operatorname{Spec}(A)\}.$$
  It follows that $\Hom_A(E',\varphi): \Hom_A(E',E_{n-1}) \to \Hom_A(E',E_{n})$ is surjective. Thus, the identity morphism on $E'$ splits off of $\varphi$ as a direct summand. In particular, $K$ is a homomorphic image of $E$, and thus $\operatorname{Supp}(K) \subseteq V_{n-1}$, as desired.
\end{proof}
\begin{corollary}
  Let $\Sigma$ be a silting complex which represents an $n$-silting module. Then $\Sigma$ is tilting.
\end{corollary}
\begin{proof}
  Clearly, the indermediate cohomology of the cosilting complex $\Sigma^+$ vanishes and thus $\Sigma^+$ represents an $n$-cosilting $A$-module. It follows that $\Sigma^+$ is a cotilting complex by Theorem~\ref{cotilting-complex}. Then $\Sigma$ is tilting by \cite[Corollary 3.5]{h24}.
\end{proof}
\begin{remark}
  In fact, \cite[Theorem 3.4]{h24} shows more: If $\Sigma$ is a silting complex over $A$ which represents an $n$-silting module then $\Sigma$ is a \textit{decent} tilting complex. This means that the heart of the silting t-structure is equivalent to the category of contramodules over the endomorphism ring of $\Sigma$ endowed with certain natural linear topology. Note that this category of contramodules, as well as the dual-analogous category of discrete modules, is then derived equivalent to $\mathrm{Mod}(A)$. See \cite{h24} for details.
\end{remark}
\subsection{Examples over local rings with vanishing intermediate cohomology}
We continue with a particular source examples of $n$-(co)silting modules over a local commutative noetherian ring $A$ which are neither $n$-(co)tilting modules nor $1$-(co)silting modules, in general.

\begin{definition}
Let $(A,\mathfrak{m},k)$ be a local ring of Krull dimension $d = \operatorname{dim}(A)>0$ and depth $s = \operatorname{depth}(A)$. We say that $A$ has \textit{vanishing intermediate local cohomology} if $H^i_\mathfrak{m}(A) = 0$ for all $s<i<d$.
\end{definition}
By a construction due to Sharp \cite[Proposition 1.4]{sha75}, such rings exist for any choice of non-negative integers $s \leq d$. An explicit example for $s=0$ and $d = 2$ is $A = k[[x,y,z]]/(x^2,xy,xz)$. We will be particularly interested in the case $s=0$, that is, the case of $A$ being of zero depth. Note that in this case, in view of Theorem~\ref{classification-noeth}, the only $n$-tilting module (resp. $n$-cotilting module) is, up to equivalence, the projective generator (resp. the injective cogenerator). However, we shall show that $n$-silting modules are much more abundant over such $A$. Furthermore, such $A$ is what is called a \textit{generalized Cohen--Macaulay} local ring, which means by definition that $H^i_\mathfrak{m}(A)$ is of finite length for all $i<d$. This is because $H^0_\mathfrak{m}(A)$ is always finite length, being a submodule of $A$, and the intermediate local cohomology modules $H^i_\mathfrak{m}(A)$ for $0<i<d$ vanish by the assumption. It follows by \cite[Theorem 37.4]{hio12} that $A$ is Cohen--Macaulay on the punctured spectrum, that is, $A_\mathfrak{p}$ is Cohen--Macaulay for any non-maximal prime ideal $\mathfrak{p}$.
\begin{example}\label{E:e1}
   Let $(A,\mathfrak{m},k)$ be a local ring of dimension $d>0$, zero depth, and with vanishing intermediate cohomology. Let $f = \operatorname{ht}$ be the height function on $\operatorname{Spec}(A)$ (defined as $\operatorname{ht(\mathfrak{p}) = \dim(A_\mathfrak{p})}$ for all $\mathfrak{p} \in \operatorname{Spec}(A)$). Let 
   $$\Sigma = \bigoplus_{\mathfrak{p} \in \operatorname{Spec}(A)}\mathbf{R}\Gamma_{V(\mathfrak{p})}(A_\mathfrak{p})[\operatorname{ht}(\mathfrak{p})]$$ 
   be the corresponding silting complex as constructed in \cite[\S 4]{hns24}. We claim that $\Sigma$ represents a $d$-silting module. For any non-maximal prime ideal $\mathfrak{p}$, the local cohomology $\mathbf{R}\Gamma_{V(\mathfrak{p})}(A_\mathfrak{p})[\operatorname{ht}(\mathfrak{p})]$ is quasi-isomorphic to a module in degree zero, this reflects $A_\mathfrak{p}$ being Cohen--Macaulay, see above. Then $H^i(\Sigma) = 0$ for all $-d<i<0$. Since $\Sigma$ is a $(d+1)$-term silting complex by \cite[\S 4]{hns24}, we see that $\Sigma$ represents a $d$-silting module. Since the monotone perversity, the function $\operatorname{ht}$ is not constant zero and thus is not bounded by depth, so that $\Sigma$ does not represent an $n$-tilting module for any $n \geq 0$. For any choice $d>1$, the monotone perversity $\operatorname{ht}$ exceeds $1$ in value, and thus $\Sigma$ is not a $2$-term silting complex, and thus does not represent a $1$-silting module.

   As per \cite[\S 5]{hns24}, note that the dual cosilting complex $\Sigma^+$ is equivalent to $\prod_{\mathfrak{p} \in \operatorname{Spec}(A)}D_{\widehat{R_\mathfrak{p}}}$, where $D_{\widehat{R_\mathfrak{p}}}$ is the dualizing complex over the $\mathfrak{p}R_{\mathfrak{p}}$-adic completion of the local ring $R_\mathfrak{p}$. Since for any non-maximal prime $\mathfrak{p}$ the local ring $R_\mathfrak{p}$ is Cohen--Macaulay, $D_{\widehat{R_\mathfrak{p}}}$ is isomorphic to the canonical module over the local ring $\widehat{R_\mathfrak{p}}$.
\end{example}

\subsection{More general construction}
The following relation between cosilting t-structures and localizations at prime ideals can be found in \cite{hhz24}. Let $f: \operatorname{Spec}(A) \to \mathbb{Z}$ be any bounded monotone perversity. For any $\mathfrak{p}$, consider $\operatorname{Spec}(A_\mathfrak{p})$ naturally as a subspace of $\operatorname{Spec}(A)$, and let $f_\mathfrak{p}: \operatorname{Spec}(A_\mathfrak{p}) \to \mathbb{Z}$ be the monotone perversity obtained by restricting $f$. Let $\Omega$ be the cosilting complex corresponding to $f$ via Theorem~\ref{classification-noeth} inducing the cosilting t-structure $(\mathcal{U},\mathcal{V})$. Then the cosilting t-structure induced by $f_\mathfrak{p}$ is of the form $(\mathcal{U}_\mathfrak{p},\mathcal{V}_\mathfrak{p})$, where
$$\mathcal{U}_\mathfrak{p} = \mathcal{U} \cap \mathbf{D}(A_\mathfrak{p}) = \{U_\mathfrak{p} \mid U \in \mathcal{U}\},$$
$$\mathcal{V}_\mathfrak{p} = \mathcal{V} \cap \mathbf{D}(A_\mathfrak{p}) = \{V_\mathfrak{p} \mid V \in \mathcal{V}\},$$
and the corresponding cosilting complex in $\mathbf{D}(A_\mathfrak{p})$ can be represented by $\Omega^\mathfrak{p} := \mathbf{R}\Hom_A(A_\mathfrak{p},\Omega)$.
\begin{construction}\label{constr}
  Let $(A,\mathfrak{m},k)$ be a local ring of dimension $d>0$, zero depth, and with vanishing intermediate cohomology. Let $f: \operatorname{Spec}(A) \to \mathbb{Z}$ be any monotone perversity satisfying the following:
  \begin{itemize}
    \item $f(\mathfrak{p}) \leq \operatorname{ht}(\mathfrak{p})$ for any non-maximal prime ideal $\mathfrak{p}$,
    \item $f(\mathfrak{m}) = d$.
  \end{itemize}
  Denote by $(\mathcal{U},\mathcal{V})$ the corresponding cosilting t-structure. We first construct the corresponding $(d+1)$-term cosilting complex explicitly, by setting
  $$\Omega = \prod_{\mathfrak{p} \in \operatorname{Spec}(A) \setminus \{\mathfrak{m}\}} C(\mathfrak{p}) \oplus D_{\widehat{A}},$$
  where:
  \begin{itemize}
    \item $C(\mathfrak{p})$ is the cotilting $A_\mathfrak{p}$ module corresponding to $f_\mathfrak{p}$,
    \item $D_{\widehat{A}}$ is the dualizing complex over the $\mathfrak{m}$-adic completion of $A$.
  \end{itemize}
  Here, note that since $A_\mathfrak{p}$ is Cohen--Macaulay for any non-maximal prime $\mathfrak{p}$, we have $\operatorname{ht}(\mathfrak{p}) = \operatorname{depth}(A_\mathfrak{p})$ and so the cosilting complex in $\mathbf{D}(A_\mathfrak{p})$ corresponding to $f_\mathfrak{p}$ via Theorem~\ref{classification-noeth} indeed represents a $\operatorname{ht}(\mathfrak{p})$-cotilting module $C(\mathfrak{p})$. Taking the standard injective resolution of $D_{\widehat{A}}$, we know that $\Omega$ is quasi-isomorphic to a complex of injectives concentrated in degrees between $0$ and $d$. Furthermore, $\Omega$ is a cogenerator of $\mathbf{D}(A)$, as it is cosupported on the whole spectrum, recall here from \cite[p. 282, 287]{har66} that $\mathbf{R}\Hom_A(k,D_{\widehat{A}}) \cong k[-d]$.

  Using Lemma~\ref{silting-complex-generate}, it thus suffices to show that $\Omega^\varkappa \in {}^{\perp_{>0}}\Omega$ for any set $\varkappa$. Recall from \cite[\S 5]{hns24} that $\Hom_{\mathbf{D}(A)}(Y,D_{\widehat{A}}[i]) = 0$ if and only if $H_\mathfrak{m}^{d-i}(Y)=0$. Thus,
  $$Y \in {}^{\perp_{>0}}D_{\widehat{A}} \iff \operatorname{depth}_\mathfrak{m}(Y) \geq d = f(\mathfrak{m}).$$
  On the other hand, we have by adjunction
  $$Y \in {}^{\perp_{>0}}C(\mathfrak{p}) \iff Y_\mathfrak{p} \in \mathcal{V}_\mathfrak{p} = \{Y \in \mathbf{D}(A_\mathfrak{p}) \mid \operatorname{depth}_\mathfrak{q}(Y) \leq f(\mathfrak{q})\}.$$

  It thus follows easily that ${}^{\perp_{>0}}\Omega = \mathcal{V}_f$. Clearly, $C(\mathfrak{p}) \in \mathcal{V}_\mathfrak{p} \subseteq \mathcal{V}$ for any prime ideal $\mathfrak{p}$ other than $\mathfrak{m}$. It remains to check that $D_{\widehat{A}} \in \mathcal{V}$, or that $\operatorname{depth} D_{\widehat{A}} \geq d$, but that follows by the aforementioned isomorphism $\mathbf{R}\Hom_A(k,D_{\widehat{A}}) \cong k[-d]$. This shows that $\Omega$ is an $(n+1)$-term cosilting complex. By the assumptions on $A$, we know that $H_\mathfrak{m}^i(A)=0$ for all $0<i<d$, but then also $H^i(D_{\widehat{A}})=0$ for all $0<i<d$, as $D_{\widehat{A}} \cong \Hom_A(\mathbf{R}\Gamma_{V(\mathfrak{m})}(A)[d],E(\mathfrak{m}))$, see e.g. \cite[(5.8)]{hns24}. It follows that $\Omega$ represents a $d$-cosilting module.
\end{construction}

A monotone perversity $f: \operatorname{Spec}(A) \to \mathbb{Z}$ is called \textit{comonotone} if $f(\mathfrak{q}) \leq f(\mathfrak{p}) +1$ whenever $\mathfrak{p} \subsetneq \mathfrak{q}$ are prime ideals such that there is no prime ideal $\mathfrak{o}$ such that $\mathfrak{p} \subsetneq \mathfrak{o} \subsetneq \mathfrak{q}$. These special perversities, under mild conditions (such as the existence of a dualizing complex) on $A$, correspond precisely to t-structures in the bounded derived category $\mathbf{D}^b(\mathrm{mod}(A))$ of finitely generated $A$-modules, see \cite{ab09}, \cite{ajs10}, \cite{tak23}. By \cite[Corollary 6.17]{pv21}, the corresponding cosilting t-structures are cotilting in this situation.

\begin{example}\label{E:e2}
  Construction~\ref{constr} easily produces examples in which the monotone perversity $f$ is not comonotone. For example, let $d>1$ and $f$ be such that $f(\mathfrak{p}) = 0$ for any non-maximal prime ideal $\mathfrak{p}$. The induced cosilting complex $\Omega$ is however still cotilting by Theorem~\ref{cotilting-complex}. The cotilting property however does not follow from the previously known cases of restrictable cosilting t-structures (or comonotone perversities) or cotilting modules.
\end{example}
\subsection{Failure of finite type for $n$-silting modules}\label{counterexample}
  In this subseciton, let 
  $$A=k[[x,y,z]]/(x^2,xy,xz)$$ 
  be the particular example of a ring of depth $0$, dimension $2$, and with vanishing intermediate cohomology $H^1_\mathfrak{m}(A)=0$, where $\mathfrak{m} = (x,y,z)$ is the maximal ideal of $A$. Note that in this specific example, $H^0_\mathfrak{m}(A)$ is isomorphic to the one dimensional vector space $k$ generated by $x$.
  
  Let now $T$ be the $2$-silting module corresponding to the $3$-term tilting complex $\Sigma$, which in turn corresponds to the function $f$ assigning $2$ to $\mathfrak{m}$ and $0$ to all other prime ideals, as per Example~\ref{E:e2}. Let $\mathfrak{p}$ be a non-maximal prime ideal of $A$. Then $\Sigma_\mathfrak{p}$ is isomorphic to a projective $A_\mathfrak{p}$-module in degree $0$ for any non-maximal prime $\mathfrak{p}$, see \cite[Lemma 6.3]{hhz24}. It follows by adjunction that for any $A_\mathfrak{p}$-module $N$ we have $N[i] \in \Sigma^{\perp_{>0}}$ for all $i>0$. In particular, we have $W \in \Sigma^{\perp_{>0}}$, where $W=\prod_{\mathfrak{p} \neq \mathfrak{m}}E(A/\mathfrak{p})$ is the product of isoclasses of indecomposable injective $A$-modules not supported on the closed point $\mathfrak{m}$.
  
  Towards a contradiction, assume that $T$ is of finite type, so that there is $\mathcal{S} \subseteq \PresFP{2}$ such that $\mathcal{D}_\Sigma = \bigcap_{\sigma \in S}\mathcal{D}_\sigma$. Clearly, this implies that there is such $\sigma$ with $H^{-2}(\sigma) \neq 0$. By the above, we have 
  $$W \in \Sigma^{\perp_{>0}} \implies W \in \mathcal{D}_\Sigma \implies W \in \mathcal{D}_\sigma \implies W \in \sigma^{\perp_{>0}}$$
  using Lemma~\ref{key lem}. Then 
  $$0=\Hom_{\mathbf{D}(A)}(\sigma,W[2]) \cong \Hom_A(H^{-2}(\sigma),W),$$ 
  which implies that the localization $H^{-2}(\sigma)_\mathfrak{p}$ at any non-maximal prime $\mathfrak{p}$ vanishes. As $\sigma  \in \PresFP{2}$, the cohomology module $H^{-2}(\sigma)$ is a submodule of a free $A$-module of finite rank, and thus $H^{-2}(\sigma)$ must be isomorphic to a finite-dimensional vector space $k^n$ over $k$. 
  
  Thus, the complex $\sigma$ witnesses tha $k^n$ is a 3rd syzygy of the zero cohomology module $M := H^0(\sigma)$. We claim that this implies that the residue field $k$ itself is a 3rd syzygy (of some finitely generated $A$-module). Note that this is not a trivial fact, as $n$th syzygies are in general not closed under direct summands. In fact, \cite[Theorem A]{addnt24} shows that over $A$, every non-trivial $3$rd syzygy has $k$ as a direct summand. Once the claim is established, the desired contratiction follows directly from \cite[Corollary 2.13]{miller}, as $A$ is manifestly not a Gorenstein ring.
  
  Set $B = A/k$, where $k$ denotes the 1-dimensional socle of $A$ generated by $x$. Following the argument of \cite[Lemma 2.8]{miller}, we have an exact sequence of the form 
  $$0 \to k^n \to A^n \to A^a \to N \to 0,$$
  where $N$ is a (1st) syzygy and the cokernel of the map $k^n \to A^n$ is isomorphic to $B^n$. Taking any one-dimensional subspace $k$ of $k^n$ induces a diagram of the following form
  $$
  \begin{CD}
    0 @>>> k^n @>>> A^n @>>> A^a @>>> N @>>> 0 \\
  &    & @AA\subseteq A @AA\subseteq A @AA = A @AA \pi A \\
    0 @>>> k @>>> A @>>> A^a @>>> L @>>> 0 \\
  \end{CD}
  $$
  The map $\pi$ is the induced projection and its kernel is easily checked using the Snake Lemma to be isomorphic to $B^{n-1}$. It follows that $L$ is an extension of $B^{n-1}$ by $N$, both of which are syzygy modules. By \cite[Theorem B]{gototak}, (1st) syzygy modules are closed under extensions for the ring $A$, so that $L$ itself is a syzygy module, which in turn makes $k$ a 3rd syzygy. Indeed, the assumptions of \textit{loc. cit.} are satisfied, as $A$ is Cohen--Macaulay at any non-maximal prime, $A$ admits a unique minimal prime ideal $(x)$, and the localization $A_{(x)} \cong k[[x,y^{\pm 1},z^{\pm 1}]]/(x^2)$ is an artinian hypersurface, and thus is Gorenstein.

\bigskip

{\bf Acknowledgements.} The first author is supported by the project LQ100192601 Lumina quaeruntur, funded by the Czech Academy of Sciences (RVO 67985840). The second author is supported by the NSF of China (Grant Nos. 12571035 and 12171206) and Jiangsu 333 Project. The third author was supported by the NSF of China (Grant No. 12201223) and the NSF of Fujian Province (No. 2023J05048).
\bigskip

\textbf{Data Availability. }Data sharing not applicable to this article as no datasets were generated or analysed during
the current study.

\bibliographystyle{abbrv}
\bibliography{references}

\end{document}